\documentclass[11pt]{article}
\usepackage[left=2.5cm,right=2.5cm,top=2.5cm,bottom=2.5cm,a4paper]{geometry}

\usepackage[a4paper]{geometry}
\usepackage{graphicx}
\usepackage{microtype}
\usepackage{siunitx}
\usepackage{booktabs}
\usepackage{cleveref}
\usepackage{graphics}
\usepackage{graphicx}
\usepackage{epsfig}
\usepackage{amsmath,amsfonts,amssymb,amsthm}
\usepackage{listings}
\usepackage{paralist}
\usepackage{sectsty}
\numberwithin{equation}{section}
\date{}
\subsectionfont{\normalfont}

\newtheorem{theorem}{Theorem}[section]

\newtheorem{lemma}[theorem]{Lemma}

\newtheorem{remark}[theorem]{Remark}

\newtheorem{question}[theorem]{Question}
\newtheorem{open problem}[theorem]{Open problem}

\providecommand{\keywords}[1]{\textbf{\textit{Key words---}} #1}

\begin{document}

\markboth{WonTae Hwang, Kyunghwan Song}
{A reciprocal sum related to the Riemann zeta function at $s=6$}

\title{A reciprocal sum related to the Riemann zeta function at $s=6$}

\author{WonTae Hwang$^{1}$, Kyunghwan Song$^{2,*}$\\
${}^{1}$Department of Mathematics, Indiana University Bloomington, 831 E. Third St.\\ Bloomington, IN 47405, USA\\
${}^{2}$Department of Mathematics, Korea University, Seoul 02841, Republic of Korea\\
${}^{*}${Corresponding Author Email: heroesof@korea.ac.kr}
}

\maketitle

\begin{abstract}
We introduce an explicit formula for a reciprocal sum related to the Riemann zeta function at $s=6,$ and pose one question related to a computational formula for larger values of $s.$


\end{abstract}

\section{Introduction}
Among various kind of zeta functions in mathematics, one of the most famous and important zeta function is the Riemann zeta function. One of the Millennium Prize Problems is the Riemann Hypothesis that is related to the non-trivial zeros of the Riemann zeta function on the critical line. For $s=\sigma + i t \in \mathbb{C}$ with $\sigma>1,$ the Riemann zeta function is defined as the absolutely convergent infinite series
\begin{equation*}
\zeta(s)=\sum_{n=1}^{\infty} \frac{1}{n^s}.
\end{equation*}
It is well-known that this function admits an analytic continuation to the whole complex plane $\mathbb{C}$,  has an Euler product formula, and satisfies a functional equation. If we restrict our attention to a small positive integer $s >1,$ then we have the following list of values of the Riemann zeta function \cite{Wol}:
\begin{equation*}
\zeta(2)=\frac{\pi^2}{6},~\zeta(3)=1.2020569032\cdots, ~\zeta(4)=\frac{\pi^4}{90},~\zeta(5)=1.0369277551\cdots,~\zeta(6)=\frac{\pi^6}{945}.
\end{equation*}
Recently, Lin\cite{Lin (2016)} initiated the study of a reciprocal sum related to $\zeta(2)$ and $\zeta(3)$, and proved the following two equalities: for any positive integer $n,$ we have
\begin{equation*}
\left[\left(\sum_{k=n}^{\infty} \frac{1}{k^2}\right)^{-1}\right]=n-1
\end{equation*}
and
\begin{equation*}
\left[\left(\sum_{k=n}^{\infty} \frac{1}{k^3}\right)^{-1}\right]=2n(n-1)
\end{equation*}
where $[x]$ denotes the greatest integer that is less than or equal to $x$. One basic observation from this result is that both $n-1$ and $2n(n-1)$ are polynomials in the variable $n.$ Lin\cite{Lin (2016)} also proposed a natural problem of determining the existence of an explicit computational formula for $\displaystyle \left[\left(\sum_{k=n}^{\infty} \frac{1}{k^s}\right)^{-1}\right]$ for an integer $s \geq 4.$ In an attempt to solve this problem, Lin and Li\cite{Lin (2017)} came up with a computational formula for the case $s=4,$ namely, they proved that
\begin{equation*}
\left[\left(\sum_{k=n}^{\infty} \frac{1}{k^4}\right)^{-1}\right]= \begin{cases} 24m^3 -18m^2 + \left[\frac{3(5m-1)}{2}\right], ~&\mbox{if}~n=2m; \\
24m^3 -54m^2 +\left[\frac{3(58m-17)}{4}\right], ~&\mbox{if}~n=2m-1.
\end{cases}
\end{equation*}
We note that their formula depends on the parity of $n.$ Along this line, Xu\cite{Xu (2016)} also proved two computational formulas related to the Riemann zeta function at $s=4, 5$, using a slightly different method from that of \cite{Lin (2017)}.  \\

In this paper, we consider the problem for the case $s=6$ and get an explicit formula for $\displaystyle \left[\left(\sum_{k=n}^{\infty} \frac{1}{k^6}\right)^{-1}\right]$, which depends on the residue of $n$ modulo $48.$ In proving some of our results, we use a different method from that of \cite{Lin (2017)} and \cite{Xu (2016)}, and we hope this method might be applied to larger values of $s$ so that the problem of Lin can be solved completely. In the Appendix, we provide another proof of the formula that is obtained for the case $s=5.$

\section{Main Result}
In this section, we prove our main result. We first introduce one notation: for an integer $n,$ let $n_{48}$ be the remainder when $n$ is divided by $48.$ Then we have the following
\begin{theorem}\label{main thm}
For each integer $n \geq 829,$ we put $\displaystyle f(n)=\left[ \left(\sum_{k=n}^{\infty}\frac{1}{k^6} \right)^{-1} \right]$. Then we have
\begin{equation*}
  f(n)=\begin{cases} 5n^5 - \frac{25}{2} n^4 +\frac{75}{4} n^3 -\frac{125}{8} n^2 +\frac{185}{48}n - \frac{5 n_{48}}{48} - \left[\frac{35-5 n_{48}}{48}\right], ~&~\mbox{if~$n$~is even} ; \\ 5n^5 - \frac{25}{2} n^4 +\frac{75}{4} n^3 -\frac{125}{8} n^2 +\frac{185}{48}n - \frac{5 n_{48}+18}{48} - \left[\frac{17-5 n_{48}}{48}\right], & ~\mbox{if~$n$~is odd}. \end{cases}
\end{equation*}
\end{theorem}
\begin{proof}
  We need to consider 48 cases according to the residue of $n$ modulo $48.$ We just give the proof of two representing cases because similar arguments can be applied to other 46 cases. In the sequel, let $c=0.999999999999999999.$ \\
  
  (i) $n \equiv 1 ~(\textrm{mod}~48)$: suppose that $n=48m+1$ for some integer $m \geq 18.$ For integers $k \geq 18$, let $p(k)=1274019840k^5 +66355200k^4 +2073600k^3 +36000k^2 +185k -1$, and let
  \begin{equation*}
    g(k)=\left(\sum_{i=1}^{48}\frac{1}{(48k+i)^6} \right) - \frac{1}{p(k)}+\frac{1}{p(k+1)},
  \end{equation*} 
  and
  \begin{equation*}
    h(k)=\left(\sum_{i=1}^{48}\frac{1}{(48k+i)^6} \right)- \frac{1}{p(k)+c}+\frac{1}{p(k+1)+c}.
  \end{equation*}
  Then by a direct computation, we can see that $g(18)<0, g(k)<g(k+1)$ for all $k \geq 18$, and $\displaystyle \lim_{k \rightarrow \infty} g(k)=0$, and hence, it follows that we have $g(k)<0$ for all $k \geq 18.$ Similarly, we also can see that $h(18)>0, h(k+1)<h(k)$ for all $k \geq 18$, and $\displaystyle \lim_{k \rightarrow \infty} h(k)=0$ so that $h(k)>0$ for all $k \geq 18.$ Therefore, we get 
  \begin{equation*}
  \frac{1}{p(k)+c}-\frac{1}{p(k+1)+c} < \sum_{i=1}^{48}\frac{1}{(48k+i)^6}  < \frac{1}{p(k)}-\frac{1}{p(k+1)}
  \end{equation*}
  for all $k \geq m.$ Hence, by summing up, we get
\begin{equation*}
\sum_{k=m}^{\infty} \left( \frac{1}{p(k)+c}-\frac{1}{p(k+1)+c} \right) < \sum_{k=m}^{\infty} \left( \sum_{i=1}^{48}\frac{1}{(48k+i)^6} \right) < \sum_{k=m}^{\infty} \left( \frac{1}{p(k)}-\frac{1}{p(k+1)} \right)
\end{equation*}
which, in turn, gives
\begin{equation*}
\frac{1}{p(m)+c}< \sum_{k=48m+1}^{\infty}\frac{1}{k^6}< \frac{1}{p(m)},
\end{equation*}
or equivalently,
\begin{equation*}
p(m)<\left(\sum_{k=48m+1}^{\infty}\frac{1}{k^6} \right)^{-1} < p(m)+c.
\end{equation*}
Now, since $p(k)$ is a polynomial with integer coefficients in $k$ so that $p(m)$ is an integer, it follows that
\begin{equation*}
f(48m+1)=\left[ \left(\sum_{k=48m+1}^{\infty}\frac{1}{k^6} \right)^{-1} \right] = p(m).
\end{equation*}
It is easy to see that this is consistent with the formula given above. \\

(ii) $n \equiv 2 ~(\textrm{mod}~48)$: suppose that $n=48m+2$ for some integer $m \geq 18.$ For integers $k \geq 18$, let $p(k)=1274019840k^5 +199065600 k^4 +13132800 k^3 +453600 k^2 +7985 k +55$, and let
  \begin{equation*}
    g(k)=\left(\sum_{i=1}^{48}\frac{1}{(48k+i+1)^6} \right) - \frac{1}{p(k)}+\frac{1}{p(k+1)},
  \end{equation*}
  and
  \begin{equation*}
    h(k)=\left(\sum_{i=1}^{48}\frac{1}{(48k+i+1)^6} \right)- \frac{1}{p(k)+c}+\frac{1}{p(k+1)+c}.
  \end{equation*}
  Then by a direct computation, we can see that $g(18)<0, g(k)<g(k+1)$ for all $k \geq 18$, and $\displaystyle \lim_{k \rightarrow \infty} g(k)=0$, and hence, it follows that we have $g(k)<0$ for all $k \geq 18.$ Similarly, we also can see that $h(18)>0, h(k+1)<h(k)$ for all $k \geq 18$, and $\displaystyle \lim_{k \rightarrow \infty} h(k)=0$ so that $h(k)>0$ for all $k \geq 18.$ Therefore, we get
  \begin{equation*}
  \frac{1}{p(k)+c}-\frac{1}{p(k+1)+c} < \sum_{i=1}^{48}\frac{1}{(48k+i+1)^6}  < \frac{1}{p(k)}-\frac{1}{p(k+1)}
  \end{equation*}
  for all $k \geq m.$ Hence, by summing up, we get
\begin{equation*}
\sum_{k=m}^{\infty} \left( \frac{1}{p(k)+c}-\frac{1}{p(k+1)+c} \right) < \sum_{k=m}^{\infty} \left( \sum_{i=1}^{48}\frac{1}{(48k+i+1)^6} \right) < \sum_{k=m}^{\infty} \left( \frac{1}{p(k)}-\frac{1}{p(k+1)} \right)
\end{equation*}
which, in turn, gives
\begin{equation*}
\frac{1}{p(m)+c}< \sum_{k=48m+2}^{\infty}\frac{1}{k^6}< \frac{1}{p(m)},
\end{equation*}
or equivalently,
\begin{equation*}
p(m)<\left(\sum_{k=48m+2}^{\infty}\frac{1}{k^6} \right)^{-1} < p(m)+c.
\end{equation*}
Now, since $p(k)$ is a polynomial with integer coefficients in $k$ so that $p(m)$ is an integer, it follows that
\begin{equation*}
f(48m+2)=\left[ \left(\sum_{k=48m+2}^{\infty}\frac{1}{k^6} \right)^{-1} \right] = p(m).
\end{equation*}
It is easy to see that this is consistent with the formula given above. \\

By applying a similar argument to other cases, we can prove the theorem.
\end{proof}
According to the proof of our theorem, we note that $f(n)$ can be expressed as a polynomial in $m$ with integer coefficients when $n=48m+b$ with $0 \leq b <48.$ \\

We conclude this section by posing one expectation:
\begin{question}\label{main ques}
Let $s \geq 7$ be an integer. Does $\displaystyle \left[ \left(\sum_{k=n}^{\infty}\frac{1}{k^s} \right)^{-1} \right]$ depend on the residue of $n$ modulo a multiple of $s-2$? If so, (for all but finitely many integers $n,$) is $\displaystyle \left[ \left(\sum_{k=n}^{\infty}\frac{1}{k^s} \right)^{-1} \right]$ a polynomial in $m$ with integer coefficients when $n=l(s-2)m+b$ with $0 \leq b < l(s-2)$ for some positive integer $l$?
\end{question}
As motivating examples, we had $s=6, l=12$ in our main result, and $s=5, l=1$ in Theorem \ref{thm_5} of the Appendix.

\begin{remark}
We note that a positive answer for Question \ref{main ques} allows us to give one possible answer for the problem of Lin. 
\end{remark}

\section{Appendix}
In this section, we give another proof of the explicit formula related to the Riemann zeta function at $s=5$, which was independently obtained by Xu\cite{Xu (2016)} and the authors. We believe that our method of proof is essentially different from that of \cite{Xu (2016)}.  First, we list three preliminary lemmas that will be used to prove the formula. We start with the following

\begin{lemma}\label{lem 1}
For any positive integer $k,$ we put $p(k)=324k^4 -216k^3 +84k^2 -16k-1.$ Then we have
\begin{equation*}
\frac{1}{p(k)+0.9}-\frac{1}{p(k+1)+0.9}< \frac{1}{(3k)^5}+\frac{1}{(3k+1)^5}+\frac{1}{(3k+2)^5}<\frac{1}{p(k)}-\frac{1}{p(k+1)}
\end{equation*}
for any $k \geq 2.$
\end{lemma}
\begin{proof}
Consider the function $g(x):=\frac{1}{(3x)^5}+\frac{1}{(3x+1)^5}+\frac{1}{(3x+2)^5}-\frac{1}{p(x)}+\frac{1}{p(x+1)}$ for $x \in [2,\infty).$ Then $g(x)$ is increasing on $[2,\infty)$ and $\displaystyle \lim_{x \rightarrow \infty} g(x)=0.$ Hence, it suffices to show that $g(2)<0.$ By a direct computation, we have
\begin{equation*}
g(2)=\frac{1}{6^5}+\frac{1}{7^5}+\frac{1}{8^5}-\frac{1}{3759}+\frac{1}{21119} \approx -0.00000006 <0.
\end{equation*}
For the other inequality, consider the function $h(x):=\frac{1}{(3x)^5}+\frac{1}{(3x+1)^5}+\frac{1}{(3x+2)^5}-\frac{1}{p(x)+0.9}+\frac{1}{p(x+1)+0.9}$ for $x \in [2,\infty).$ Then $h(x)$ is decreasing on $[2,\infty)$ and $\displaystyle \lim_{x \rightarrow \infty} h(x)=0.$ Hence, it suffices to show that $h(2)>0.$ By a direct computation, we have
\begin{equation*}
h(2)=\frac{1}{6^5}+\frac{1}{7^5}+\frac{1}{8^5}-\frac{1}{3759.9}+\frac{1}{21119.9} \approx 0.000000001 >0.
\end{equation*}
This completes the proof.
\end{proof}

Similarly, we have two more related results:

\begin{lemma}\label{lem 2}
For any positive integer $k,$ we put $q(k)=324k^4 +216 k^3 +84k^2 +16k -1.$ Then we have
\begin{equation*}
\frac{1}{q(k)+0.99}-\frac{1}{q(k+1)+0.99}< \frac{1}{(3k+1)^5}+\frac{1}{(3k+2)^5}+\frac{1}{(3k+3)^5}<\frac{1}{q(k)}-\frac{1}{q(k+1)}
\end{equation*}
for any $k \geq 1.$
\end{lemma}
\begin{proof}
Consider the function $g(x):=\frac{1}{(3x+1)^5}+\frac{1}{(3x+2)^5}+\frac{1}{(3x+3)^5}-\frac{1}{q(x)}+\frac{1}{q(x+1)}$ for $x \in [1,\infty).$ Then $g(x)$ is increasing on $[1,\infty)$ and $\displaystyle \lim_{x \rightarrow \infty} g(x)=0.$ Hence, it suffices to show that $g(1)<0.$ By a direct computation, we have
\begin{equation*}
g(1)=\frac{1}{4^5}+\frac{1}{5^5}+\frac{1}{6^5}-\frac{1}{639}+\frac{1}{7279} \approx -0.000002 <0.
\end{equation*}
For the other inequality, consider the function $h(x):=\frac{1}{(3x+1)^5}+\frac{1}{(3x+2)^5}+\frac{1}{(3x+3)^5}-\frac{1}{q(x)+0.99}+\frac{1}{q(x+1)+0.99}$ for $x \in [1,\infty).$ Then $h(x)$ is decreasing on $[1,\infty)$ and $\displaystyle \lim_{x \rightarrow \infty} h(x)=0.$ Hence, it suffices to show that $h(1)>0.$ By a direct computation, we have
\begin{equation*}
h(1)=\frac{1}{4^5}+\frac{1}{5^5}+\frac{1}{6^5}-\frac{1}{639.99}+\frac{1}{7279.99} \approx 0.000000001 >0.
\end{equation*}
This completes the proof.
\end{proof}

\begin{lemma}\label{lem 3}
For any positive integer $k,$ we put $r(k)=324k^4 +648k^3 +516k^2 +192k +26.$ Then we have
\begin{equation*}
\frac{1}{r(k)+0.9}-\frac{1}{r(k+1)+0.9}< \frac{1}{(3k+2)^5}+\frac{1}{(3k+3)^5}+\frac{1}{(3k+4)^5}<\frac{1}{r(k)}-\frac{1}{r(k+1)}
\end{equation*}
for any $k \geq 1.$
\end{lemma}
\begin{proof}
Consider the function $g(x):=\frac{1}{(3x+2)^5}+\frac{1}{(3x+3)^5}+\frac{1}{(3x+4)^5}-\frac{1}{r(x)}+\frac{1}{r(x+1)}$ for $x \in [1,\infty).$ Then $g(x)$ is increasing on $[1,\infty)$ and $\displaystyle \lim_{x \rightarrow \infty} g(x)=0.$ Hence, it suffices to show that $g(1)<0.$ By a direct computation, we have
\begin{equation*}
g(1)=\frac{1}{5^5}+\frac{1}{6^5}+\frac{1}{675}-\frac{1}{1706}+\frac{1}{12842} \approx -0.0000001 <0.
\end{equation*}
For the other inequality, consider the function $h(x):=\frac{1}{(3x+2)^5}+\frac{1}{(3x+3)^5}+\frac{1}{(3x+4)^5}-\frac{1}{r(x)+0.9}+\frac{1}{r(x+1)+0.9}$ for $x \in [1,\infty).$ Then $h(x)$ is decreasing on $[1,\infty)$ and $\displaystyle \lim_{x \rightarrow \infty} h(x)=0.$ Hence, it suffices to show that $h(1)>0.$ By a direct computation, we have
\begin{equation*}
h(1)=\frac{1}{5^5}+\frac{1}{6^5}+\frac{1}{7^5}-\frac{1}{1706.9}+\frac{1}{12842.9} \approx 0.0000002 >0.
\end{equation*}
This completes the proof.
\end{proof}

Now, we are ready to prove the formula. Let $p(k), q(k),$ and $r(k)$ be the polynomials as in the three lemmas given above.

\begin{theorem}\label{thm_5}
For each integer $n \geq 4,$ we put $\displaystyle f(n)=\left[ \left(\sum_{k=n}^{\infty}\frac{1}{k^5} \right)^{-1} \right]$. Then we have
\begin{equation*}
f(n)=\begin{cases} p(m), ~&\mbox{if}~n=3m; \\
q(m), ~&\mbox{if}~n=3m+1; \\
r(m), ~&\mbox{if}~n=3m+2.
\end{cases}
\end{equation*}
\end{theorem}
\begin{proof}
Suppose first that $n=3m$ for some $m \geq 2.$ By Lemma \ref{lem 1}, we know that
\begin{equation*}
\frac{1}{p(k)+0.9}-\frac{1}{p(k+1)+0.9}< \frac{1}{(3k)^5}+\frac{1}{(3k+1)^5}+\frac{1}{(3k+2)^5}<\frac{1}{p(k)}-\frac{1}{p(k+1)}
\end{equation*}
for each $k \geq m.$ Hence, by summing up, we get
\begin{equation*}
\sum_{k=m}^{\infty} \left( \frac{1}{p(k)+0.9}-\frac{1}{p(k+1)+0.9} \right) < \sum_{k=m}^{\infty} \left( \frac{1}{(3k)^5}+\frac{1}{(3k+1)^5}+\frac{1}{(3k+2)^5} \right) <
\end{equation*}
\begin{equation*}
\sum_{k=m}^{\infty} \left( \frac{1}{p(k)}-\frac{1}{p(k+1)} \right)
\end{equation*}
which, in turn, gives
\begin{equation*}
\frac{1}{p(m)+0.9}< \sum_{k=3m}^{\infty}\frac{1}{k^5}< \frac{1}{p(m)},
\end{equation*}
or equivalently,
\begin{equation*}
p(m)<\left(\sum_{k=3m}^{\infty}\frac{1}{k^5} \right)^{-1} < p(m)+0.9.
\end{equation*}
Now, since $p(k)$ is a polynomial with integer coefficients in $k$ so that $p(m)$ is an integer, it follows that
\begin{equation*}
f(3m)=\left[ \left(\sum_{k=3m}^{\infty}\frac{1}{k^5} \right)^{-1} \right] = p(m).
\end{equation*}

Now, suppose that $n=3m+1$ for some $m \geq 1.$ By Lemma \ref{lem 2}, we know that
\begin{equation*}
\frac{1}{q(k)+0.99}-\frac{1}{q(k+1)+0.99}< \frac{1}{(3k+1)^5}+\frac{1}{(3k+2)^5}+\frac{1}{(3k+3)^5}<\frac{1}{q(k)}-\frac{1}{q(k+1)}
\end{equation*}
for each $k \geq m.$ Hence, by summing up, we get
\begin{equation*}
\sum_{k=m}^{\infty} \left( \frac{1}{q(k)+0.99}-\frac{1}{q(k+1)+0.99} \right) < \sum_{k=m}^{\infty} \left( \frac{1}{(3k+1)^5}+\frac{1}{(3k+2)^5}+\frac{1}{(3k+3)^5} \right) <
\end{equation*}
\begin{equation*}
\sum_{k=m}^{\infty} \left( \frac{1}{q(k)}-\frac{1}{q(k+1)} \right)
\end{equation*}
which, in turn, gives
\begin{equation*}
\frac{1}{q(m)+0.99}< \sum_{k=3m+1}^{\infty}\frac{1}{k^5}< \frac{1}{q(m)},
\end{equation*}
or equivalently,
\begin{equation*}
q(m)<\left(\sum_{k=3m+1}^{\infty}\frac{1}{k^5} \right)^{-1} < q(m)+0.99.
\end{equation*}
Now, since $q(k)$ is a polynomial with integer coefficients in $k$ so that $q(m)$ is an integer, it follows that
\begin{equation*}
f(3m+1)=\left[ \left(\sum_{k=3m+1}^{\infty}\frac{1}{k^5} \right)^{-1} \right] = q(m).
\end{equation*}

Finally, suppose that $n=3m+2$ for some $m \geq 1.$ By Lemma \ref{lem 3}, we know that
\begin{equation*}
\frac{1}{r(k)+0.9}-\frac{1}{r(k+1)+0.9}< \frac{1}{(3k+2)^5}+\frac{1}{(3k+3)^5}+\frac{1}{(3k+4)^5}<\frac{1}{r(k)}-\frac{1}{r(k+1)}
\end{equation*}
for each $k \geq m.$ Hence, by summing up, we get
\begin{equation*}
\sum_{k=m}^{\infty} \left( \frac{1}{r(k)+0.9}-\frac{1}{r(k+1)+0.9} \right) < \sum_{k=m}^{\infty} \left( \frac{1}{(3k+2)^5}+\frac{1}{(3k+3)^5}+\frac{1}{(3k+4)^5} \right) <
\end{equation*}
\begin{equation*}
\sum_{k=m}^{\infty} \left( \frac{1}{r(k)}-\frac{1}{r(k+1)} \right)
\end{equation*}
which, in turn, gives
\begin{equation*}
\frac{1}{r(m)+0.9}< \sum_{k=3m+2}^{\infty}\frac{1}{k^5}< \frac{1}{r(m)},
\end{equation*}
or equivalently,
\begin{equation*}
r(m)<\left(\sum_{k=3m+2}^{\infty}\frac{1}{k^5} \right)^{-1} < r(m)+0.9.
\end{equation*}
Now, since $r(k)$ is a polynomial with integer coefficients in $k$ so that $r(m)$ is an integer, it follows that
\begin{equation*}
f(3m+2)=\left[ \left(\sum_{k=3m+2}^{\infty}\frac{1}{k^5} \right)^{-1} \right] = r(m).
\end{equation*}
This completes the proof of the theorem.
\end{proof}
We conclude this paper by mentioning that the idea of the proof of Theorem \ref{thm_5} was used to prove Theorem \ref{main thm}.

\end{document}